\pdfoutput=1
\documentclass[11pt,reqno]{amsart}
\usepackage[paperheight=279mm,paperwidth=18cm,textheight=26cm,textwidth=14cm,includehead]{geometry}
\usepackage[T1]{fontenc}
\usepackage[utf8]{inputenc}
\usepackage[unicode]{hyperref}
\hypersetup{
bookmarks=true,
colorlinks=true,
citecolor=[rgb]{0,0,0.5},
linkcolor=[rgb]{0,0,0.5},
urlcolor=[rgb]{0,0,0.75},
pdfpagemode=UseNone,
pdfstartview=FitH,
pdfdisplaydoctitle=true,
pdftitle={Bi-parameter Carleson embeddings with product weights},
pdfauthor={Arcozzi, Mozolyako, Psaromiligkos, Volberg, Zorin-Kranich},
pdflang=en-US
}

\usepackage[style=alphabetic,maxalphanames=4]{biblatex}
\usepackage{amssymb}
\usepackage{amsmath}
\usepackage{amsthm}
\usepackage{amsfonts}
\usepackage{mathtools}

\usepackage{tikz}
\usetikzlibrary{patterns}

\DeclarePairedDelimiterX\Set[1]\{\}{%

#1
}

\DeclarePairedDelimiterX\innerp[2]{\langle}{\rangle}{#1,#2}

\numberwithin{equation}{section}
\theoremstyle{plain}
\newtheorem{theorem}[equation]{Theorem}

\newtheorem{lemma}[equation]{Lemma}
\newtheorem{corollary}[equation]{Corollary}
\newtheorem{conjecture}[equation]{Conjecture}

\theoremstyle{remark}

\theoremstyle{definition}

\newcommand{\const}{{\rm const}}

\newcommand{\pd}{\partial}
\def\dist{\operatorname{dist}}

\newcommand{\al}{\alpha}
\newcommand{\eps}{\varepsilon}
\newcommand{\la}{\lambda}

\newcommand{\om}{\omega}
\newcommand{\Om}{\Omega}

\newcommand{\cU}{\mathcal U}
\newcommand{\cV}{\mathcal V}
\newcommand{\cH}{\mathcal H}

\newcommand{\bC}{\mathbb C}
\newcommand{\bD}{\mathbb D}

\newcommand{\bR}{\mathbb R}

\newcommand{\bZ}{\mathbb Z}

\begin{document}

\title[One phase problem]%
{One phase problem for two positive harmonic function: below the codimension $1$ threshold}
\author[A.~Volberg]{Alexander Volberg}
\thanks{AV is partially supported by the NSF grants  DMS-1900268 and DMS-2154402}
\address[A.~Volberg]{Department of Mathematics, Michigan Sate University,  and Hausdorff Center, Universit\"at Bonn}
\email{volberg@math.msu.edu}
%
%
\begin{abstract} What can  be said  about the domain $\Om$ in $\bR^n$ for which its Green's function $G(z)$  satisfies $G(z)\asymp \dist (z, \pd\Om)^\delta$? What can we say about $\Om$ if the Boundary Harnack Principle holds in the form $u/v=\text{real analytic}$ on the part $E$ of its boundary? Here $u, v$ are positive harmonic functions on $\Om$ vanishing on $E$. Is this part of the boundary also nice?
 We discuss these questions below and give answers in very special cases.
\end{abstract}
\maketitle

\section{Preliminaries and notations}
\label{prelim}
In recent years several outstanding problems at the edge of harmonic analysis and geometric measure theory in $\bR^n$ were solved, see e.g. \cite{ENV}, \cite{HMM}, \cite{NTV}, \cite{DM}, \cite{DFM} and the literature therein. However, these works revealed a new threshold: the technique works if the codimension of a (very irregular) set in question  has codimension at most $1$ in $\bR^n$. 

Here we present several questions (some of them  looking very simple) and a tiny bit of answers  breaking down this threshold. We consider
only $n=2$ case, but most of the problems (and some solutions) presented below work for all $n$.

Two phase problem for positive harmonic function can be vaguely described as follows: there are two disjoint domains $\cU, \cV$ in $\bR^n$ with a common part of their boundaries, $E$ called interface. Let $O$ be an open neighborhood of $E$.There are two positive harmonic functions: $u$ in $O\cap \cU$ and $v$ in $O\cap \cV$. Both functions vanish on $E$ (in some sense), but there ratio is well defined on $E$ (again in some sense) and this ratio is a nice function on $E$. One then concludes that  then $E$ is nice (in some sense).  The example of this kind of problem can be found in \cite{KT1}, \cite{AMTV}, where $u=G_\cU(\cdot, p), v=G_\cV(\cdot, q)$, $p\in \cU, q\in \cV$, two Green functions of corresponding domains and the assumption on the ratio of $u$ and $v$ is formulated in terms of
harmonic measures $\om_\cU(\cdot, p), \om_\cV(\cdot, q)$ of two disjoint domains being mutually absolutely continuous (and non-zero) on the interface $E$. 

On the other hand, a one phase problem for positive harmonic function defined near the boundary and  vanishing at the boundary (or for harmonic measure) usually compares the harmonic measure of one given domain with the Hausdorff measure on the boundary. The literature here is rather vast, see, for example, \cite{AHM3TV}, \cite{KT2}, \cite{KT3}, \cite{KPT} and citation therein.

\medskip

We wish to consider a one phase problem in domain $\Om$, with two positive harmonic functions $u, v$ in $O\cap \pd\Om$, where $O$ is just an open disc (ball), $u=0 =v$ on  $E:=O\cap \pd\Om$.  The main assumption is that the ratio  $\frac{u}{v}$ makes a rather nice function on $E$. We would wish to prove that then $E$ is nice. One can think about these two harmonic functions as two Green's functions of the domain $\Om$ but with different poles. Then we are considering the comparison of harmonic measures related to two different poles, rather than comparison of harmonic measure and Hausdorff measure on the boundary.

\smallskip

Now we need to explain what ``nice'' means and what kind of domains $\Om\subset \bR^n$ we consider.
Let us remind that if $\Om$ is Lipschitz domain then $\log\frac{u}{v}$ is H\"older continuous on $E=B(x, r)\cap \Om, x\in \pd\Om,$ by Jerison--Kenig theorem \cite{JK}. Moreover, by their result the same is true for NTA domains, see the definition in \cite{JK}.

We will deal mostly (but not exclusively) with plane domains, $n=2$, but our $\Om$ can very well be infinitely connected, with rather complicated (but compact) boundary. However, we will require several things from $\Om$ to make it ``one sided NTA domain''.

So, we require that $\pd\Om$ has 1) capacity density condition (CDC): $\text{cap}(B(z, \delta))\asymp \text{cap} (B(z, \delta)\cap \pd\Om)$ for all small $\delta$ and all $z\in \pd\Om$; 2) corkscrew condition, see \cite{JK}; 3) Harnack chain condition, see \cite{JK}.

The NTA domains of course satisfy 1), 2), 3). But another example that is important in what follows is the domain $\Om$ whose boundary $\pd\Om$ is a sufficiently nice Cantor set, for example a self-similar set in the sense of conformal dynamical systems.

\medskip

Let us explain briefly what is conformal self-similar sets (conformal repellers). We consider more general sets below, but to have in mind those special examples is  a good idea. So, Let $O$ be a topological discs and $O_1,\dots, O_d$, $d\ge 2$, be topological discs with disjoint closures inside $O$, we require also that $\pd O$ were disjoint with closures of all $O_j$. Let $f=(f_1, \dots, f_d)$ be a collection of conformal maps of $O_i$ onto $O$. By considering $f_i^{-1}(O_j)$, $i, j=1,\dots, d$ we see $d^2$ topological discs of the second generation ($d$ groups of $d$ discs each) that are mapped onto $O$ conformally by iterative power $f^2= f\circ f$ (think about $f$ as a piecewise conformal map). By iterating this construction we have $d^3$ discs of the third generation, et cetera. Finally we can consider the unions of discs of generation $k$, call ot $U_k$ and $J=\cap_k U_k$. Domain $\Om=\bC\setminus J$ is a typical example of domains we consider here, and $O$ itself is atypical example of a neighborhood of $J=\pd \Om$.

\medskip

It is easy to prove that such one sided NTA domains satisfy Jerison--Kenig boundary Harnack principle (BHP), meaning that for any two harmonic and positive $u, v$ in $\Om\cap B(z, r)$ ($z\in \pd\Om$, $r$ being small) that vanish on $ B(z, r)\cap \pd\Om$, one has the H\"older property for $u/v$ on $ B(z, r)\cap \pd\Om$ in the following sense:
\begin{equation}
\label{BHP}
\exists \eps\in (0,1):\forall z_1, z_2 \in \Om \cap B(z, r)\quad \Big|\log\frac{u(z_1)}{v(z_1)} - \log\frac{u(z_2)}{v(z_2)}\Big|\le C|z_1-z_2|^\eps\,.
\end{equation}
Tending $z_i$ to $\zeta_i\in   B(z, r)\cap\pd\Om$, $i=1,2$ we get the H\"older property of $u/v$ on $B(z, r)\cap\pd\Om$ (which is called BHP).
By the way, we use complex notations because below we deal mostly with planar $\Om\subset \bR^2$, but all our questions and conjectures make sense for $n>2$ as well.

\bigskip

Now let us explain what we mean by saying that $\frac{u}{v}$ is ``nice" on $B(z, r)\cap\pd\Om$. The above \eqref{BHP} always happen on one sided NTA, so it is normal and not sufficiently nice. By ``nice" we understand the case a much stronger property:
when  in one sided  NTA domain described by 1), 2), 3) above, on the top of \eqref{BHP}, we have also
\begin{equation}
\label{nicera}
\exists\, \text{real analytic}\,\, R\,\, \text{on}\,\, B(z, r), z\in \pd\Om: \forall \zeta\in B(z, r)\cap\pd\Om,\quad \frac{u(\zeta)}{v(\zeta)} = R(\zeta)\,.
\end{equation}

\bigskip

We now list several conjectures in the order of decreasing difficulty.

\begin{conjecture}
\label{conj1}
Let $\Om$ be our one sided NTA, $ u, v$ be as above, in particular, let \eqref{nicera} holds. Then either $u=\la v$ for some constant $\la$ or $B(z, r)\cap \pd\Om$ is real analytic maybe with the exception of a set of dimension $n-2$.
\end{conjecture}

\begin{conjecture}
\label{conj2}
Let $n=2$. Let $\Om$ be our one sided NTA, $u, v$ be as above, in particular, let \eqref{nicera} hold, but let also $R= |A|$, where $A$ is a  non-constant holomorphic function on $B(z, r)$. Then either $u=\la v$ for some constant $\la$ or  $B(z, r)\cap \pd\Om$ is real analytic maybe with the exception of a set of isolated points.
\end{conjecture}

This latter conjecture is proved in \cite{VaVo}, \cite{VaVo1} for the case of simply connected $\Om$. The paper uses a very powerful  result of Sakai \cite{Sa} that describes the so-called Schwarz functions in arbitrary planar domains. However, the reduction of Conjecture \ref{conj2} to Sakai's theorem works only in simply (or finitely) connected domains. Any solution of Conjecture \ref{conj2} should first eliminate the possibility of infinitely connected domain. For the simply connected case Conjecture \ref{conj2} turned out to be closely related to the so-called Nevanlinna domains, which were extensively studied e.g. in \cite{DK},  \cite{BaF}, \cite{BeF}, \cite{BeBoF}, \cite{MM} and the literature cited therein.

\begin{conjecture}
\label{conj3}
Let $n=2$. Let Let $\Om$ be our one sided NTA, $ u, v$ be as above, in particular, let \eqref{nicera} hold, but let also $R= |A|$, where $A$ is a non-constant holomorphic function on $B(z, r)$.  Let, in addition, $\pd\Om$ be a regular Cantor set in the following sense:
\begin{equation}
\label{regular}
\exists a>1:\, \forall k\in \bZ_+\,\, \text{the set}\,\, \{z:\dist (z, \pd\Om)< a^{-k}\}= V_1\cup\dots V_{m_k},
\end{equation}
where $\text{length}(V_i )\le C a^{-k}$ and $m_k\le C a^{\delta k}$, for some $\delta>0$.
Then  $u=\la v$ for some constant $\la$.
\end{conjecture}

The conclusion is natural. Cantor structure of the boundary unequivocally says that boundary is {\it not} nice and definitely is {\it not} real analytic. Therefore, the only option that should exist for $u, v$ is to be proportional.

\begin{conjecture}
\label{conj4}
Let $n=2$. Let $\Om$ be our one sided NTA, $u, v$ be as above, in particular, let \eqref{nicera} hold, but let also $R= |A|$, where $A$ is a  holomorphic function on $B(z, r)$.  Let, in addition, $\Om$ be such that its Green's function (with some pole) satisfies
\begin{equation}
\label{green}
c_1\dist(z, \pd\Om)^\delta \le G(z, p) \le c_2\dist(z, \pd\Om)^\delta,\quad \text{for some}\,\, \delta\in (0,1]
\end{equation} 
for all $z\in \Om$ sufficiently close to $\pd\Om$ and for some $\delta>0$.
Then  either $u=\la v$ for some constant $\la$ or $\delta=1$ and $B(z, r)\cap \pd\Om$ is piecewise real analytic.
\end{conjecture}

It is easy to see that assumption \eqref{green} (if $\delta<1$) implies that $\pd\Om$ is  a regular Cantor set if $\pd\Om$ is a conformal self-similar set, see \cite{AV} . So the proof of Conjecture \ref{conj3} implies Conjecture \eqref{conj4} at least for conformal self-similar sets. But we want to formulate a very ``easy" conjecture that claims  part of Conjecture \ref{conj4}. 
The following statement looks very easy, but I can prove it only in one special case (see below). It is interesting in all  dimensions $n\ge 2$.

\begin{conjecture}
\label{conj5}
Let for $\Om\subset \bR^n$ we have \eqref{green} with $\delta>0$. Then 1) $\delta=1$ necessarily; 2) $\pd\Om$ has a certain weak smoothness, at least it is rectifiable.
\end{conjecture}

Those conjectures appear naturally in a certain unsolved problem in complex dynamics. Conjecture \ref{conj5} is especially intriguing because it looks so simple. Papers \cite{DM} and \cite{DFM} consider a very closely related problem of prevalent approximation of Green's function by the distance. But for co-dimensions bigger than one cases the Green's function is pertinent to a certain degenerate elliptic equation as far as I understand. Still the technique of these papers can probably be useful for solving this conjecture.

\section{Warming up}
\label{warm}
To warm up we consider here Conjecture \ref{conj5} in an extremely special case. We assume that \eqref{green} holds with $\delta=1$:
\begin{equation}
\label{green1}
c_1\dist(z, \pd\Om) \le G(z, p) \le c_2\dist(z, \pd\Om),
\end{equation} 
but $\pd\Om$ is a conformal self-similar Cantor set (conformal repeller)  of Hausdorff dimension $1$ in the sense of complex dynamics, for example, the reader can think that it is a $1/4$  corner Cantor set. We wish to lead \eqref{green1} to contradiction.  For brevity we adopt the notation
$$
J:=\pd\Om\,.
$$

Notice that \eqref{green1} immediately implies $\pd G|\le C$, hence the following Cauchy integral is bounded:
\begin{equation}
\label{cauchy}
C^\om(z):=\int_{J} \frac{d\om(\zeta)}{z-\zeta}  \in L^\infty (\bC\setminus J)\,.
\end{equation}
From here we can standardly  deduce  that maximal Cauchy integral
\begin{equation}
\label{star}
C^\om_*(z) := \sup_{r>0} \int_{J\setminus b(z, r)}\Big| \frac{d\om(\zeta)}{z-\zeta}\Big| \le C\quad \forall z\in J\,.
\end{equation}

In fact,  \eqref{star} follows from \eqref{cauchy} by mean value theorem for subharmonic functions if we use one more interesting inequality about harmonic measure (the pole is immaterial) of $\Om$:
\begin{equation}
\label{om}
\om (B(z, r)) \asymp r, \quad \forall z\in J\,.
\end{equation}
The latter inequality follows from CDC assumption on $J$, from \eqref{green} and from corkscrew assumption on $\Om$. Such folklore estimates can be found in many places, e.g. in \cite{AHM3TV}.

\medskip

The boundedness of Maximal Cauchy operator obtained in \eqref{star} implies that the Cauchy integral operator $f\to C^{fd\om}$ is a bounded operator in $L^2(\om)$.  But then \eqref{om} implies that he Cauchy integral operator $f\to C^{fd\cH^1}$ is a bounded operator in $L^2(\cH^1)$, where $\cH^1$ is the Hausdorff measure of dimension $1$ (length):
\begin{equation}
\label{op}
\int_J\Big|\int_J \frac{f (\zeta)d\cH^1}{z-\zeta}\Big|^2 d\cH^1(z) \le C\int_J |f(\zeta)|^2 d\cH^1(\zeta),\quad \forall f\in L^2(J, \cH^1)\,.
\end{equation}
Melnikov found out a beautiful symmetrization trick (see \cite{XT}) that deduces from \eqref{op} that Menger's curvature $c_2(\mu)$ of measure $\mu=d\cH^1|J$ is finite meaning that
\begin{equation}
\label{Mc}
c_2^2(\mu):=\int_J\int_J\int_J \frac{d\mu(z_1) s\mu(z_2) d\mu(z_3)}{R^2(z_1, z_2, z_3)} <\infty,
\end{equation}
where $R(z_1, z_2, z_3)$ is the radius of a circle passing through $z_1, z_2, z_3$.

It is left to refer to \cite{XT}:  for $1/4$ corner Cantor set this Menger's curvature is infinite. Exactly the same proof shows that Menger's curvature is infinite for any dimension $1$ conformal self-similar Cantor set (conformal repeller). 

\medskip

So Conjecture \ref{conj5} is proved for a very special case. And the proof uses several rather sophisticated results.
There should be easier proof that works in general situation. We are unaware of it.

\smallskip

However, it is interesting to notice that basically the same claim and the same proof will work for domains in $\bR^n$, $n>2$.  Let us  give a brief explanation: of course
Melnikov's symmetrization and his reduction to  Menger's curvature  does not work in $\bR^n$, $n>2$. This is a very well-known difficulty.
But, for example, paper \cite{NTV} would circumvent this difficulty. The fact is that the solution of David--Semmes conjecture implies that
domain satisfying \eqref{green1} should be rectifiable, the proof will be exactly as above, but instead of Cauchy transform, one would need to use singular Riesz transforms of singularity $n-1$ and use \cite{NTV}. One can be referred to \cite{JA}, where another proof is given.

\medskip

\begin{corollary}[Carleson, 1985]
\label{Carl}
Let $J$ be conformal self-similar Cantor set (conformal repeller)  of dimension $1$. Then the dimension of harmonic measure of $\Om:=\bC\setminus J$ is strictly less than $1$: $\dim \om <1$.
\end{corollary}
\begin{proof}
First one can use the theory of Gibbs measures (see \cite{RB}) and notice that on conformal repeller the Hausdorff measure $\cH^1|J$ and $\om$ are Gibbs measures (meaning that their Jacobians are H\"older continuous on $J$). So they are either boundedly mutually absolutely continuous with bounded density or they are mutually singular. In the former case one has \eqref{om}, and, hence, \eqref{green1}, which brings us to contradiction--see above. In the case of singularity,  we consider the invariant ergodic measure $\nu$ equivalent to $\om$, and then variational principle for Gibbs measure $\cH^1|J$ (or rather for its invariant counterpart) implies that
$$
h_\nu-\int_J\log |f'(\zeta)| d\nu(\zeta)<0,
$$
where $f$ is the conformal dynamical system that created $J$ and $h_\nu$ is the entropy of $\nu$. By Manning's formula \cite{AM}  the previous inequality implies
$$
\dim \om=\dim\nu = \frac{h_\nu}{\int_J\log |f'(\zeta)| d\nu(\zeta)}<1\,.
$$
\end{proof}
Carleson's proof \cite{LC} was much more ``hands on". The above proof is in \cite{AV} with more details.

\medskip

The following statement would be an extension of the result of Carleson, but  it is still a conjecture. It would be solved if Conjecture \ref{conj4} were solved.
\begin{conjecture}
\label{conj6}
For any conformal self-similar set $J$ of dimension $\delta<1$ its harmonic measure has strictly smaller dimension:
$$
\dim \om_{\bC\setminus J} <\delta\,.
$$
\end{conjecture}
We intentionally restricted to the case $\delta<1$. In truth, if $\delta=1$ this is proved by Carleson, see above, and if $\delta>1$, one just uses a celebrated result of Jones--Wolff  \cite{JW} that for any planar domain harmonic measure $\dim\om\le 1$.  So, we would get $\dim\om <\dim J$ for all possible conformal self-similar sets $J$. In higher dimensions $n$ one can still compare $\dim\om $ and $\dim\pd\Om$  and
\cite{JA} proves $\dim \om <\dim \pd \Om$ when $\pd\Om$ is $s$-Ahlfors--David regular, $n-1\le s\le n$, and $\pd\Om$ is ``uniformly non-flat''.
See the definition of uniformly non-flat in ]cite{JA}, fractal boundaries are uniformly non-flat.

\section{The proof of Conjecture \ref{conj4} in a special case}
\label{R}

\begin{theorem}
\label{Rth}
Let $n=2$. Let $\Om$ is a one sided  NTA domain as above and let $\pd\Om$ be a regular Cantor set. Let $n=2$. Let  \eqref{nicera} hold. Let, in addition, $\Om$ be such that its Green's function (with some pole) satisfies
\begin{equation}
\label{green2}
c_1\dist(z, \pd\Om)^\delta \le G(z, p) \le c_2\dist(z, \pd\Om)^\delta
\end{equation} 
for all $z\in \Om$ sufficiently close to $\pd\Om$ and for some $\delta>0$. In addition we require 
\begin{equation}
\label{Ra}
\pd\Om\subset \bR\,.
\end{equation}
Then $u=\la v$ for some constant $\la$.
\end{theorem}

The assumption \eqref{Ra} is decisively restrictive and I do not know how to get rid of it.
\begin{proof}
For brevity again denote $J=\pd\Om$. We may think that $J\subset [-1,1]=:I$.
First use of \eqref{Ra} is that $R|I =A|I>0$ and $A|I$ is the trace of analytic function in the neighborhood of $I$.
Of course $A=u/v>0$ on $J$, which of course means $A$ is symmetric. 
We can write
$$
A= e^{\al+i\tilde\al}\,.
$$
The second use of \eqref{Ra} (but not the last one) we can think that $u$ and $v$ are symmetric:
$$
u(z) =\overline{u(\bar z)};\,\, v(z) =\overline{v(\bar z)}\,.
$$

\begin{lemma}
\label{L}
Let $d(z):= \dist (z, J)$.  Let $U$ be a neighborhood of $J$. Regularity of $J$ gives that there exists $\delta>0$ such that
 $$
 \frac1{d^{1-\delta}} \in  L^{2+\delta}(U, dm_2)\,.
 $$
 \end{lemma}
 \begin{proof}
 Let $D_k=\{ z: d(z) \le a^{-k}\}$. Then $D_k= V_1\cup\dots\cup V_{m_k}, \quad  m_k \le  a^{k\delta}$.
 \begin{eqnarray*}
 \label{1}
& \int \Big(\frac1{d^{1-\delta}}\Big)^{2+\delta}\, dm_2 =\sum_{k=0}^\infty \sum_{i=1}^{m_k} m_2(V_i)\Big(\frac1{a^{-k(1-\delta)}}\Big)^{2+\delta} \le
 \\
 & \sum_k m_k \cdot a^{2k}Big(\frac1{a^{-k(1-\delta)}}\Big)^{2+\delta} \le \sum_k a^{k\delta}\cdot 
 a^{-2k}\cdot a^{k(2+\delta -2\delta -\delta^2} =
 \\
 &\sum_k a^{-k\delta^2}<\infty.
 \end{eqnarray*}
 
 \end{proof}
 
 \bigskip
 
 We wish to prove that $\pd\Phi\in C^{1+\frac{\delta}{2+\delta}} $. This will require a bit of work.
 We know $A=e^{\al+i\tilde \al}$ and $ \frac{u}{v}= A=|A| =e^\al$ on $J$. Then using \eqref{BHP} and \eqref{green2} we obtain
 \begin{equation}
 \label{better}
 |e^\al v-u|=v\Big| e^\al -\frac{u}{v}\Big| \lesssim d^{\delta+\eps}\,.
 \end{equation}
 
 \begin{lemma}
 \label{2}
 Let $U$ be a neighborhood of $J$ and $\al$ be a harmonic function on $U$. Let $u, v$ be two harmonic functions in $U\setminus J$ and $|u|+|v|\lesssim d^\delta$, and $e^\al v-u|\lesssim d^{\gamma}$ with $\gamma-1 \le \delta$. Then
 $$
 |\nabla (e^{\al} v -u)| \lesssim d^{\gamma-1}\,.
 $$
 \end{lemma}
 \begin{proof}
 Fix $z_0\in U\setminus J$, let $d_0=d(z_0)$. Consider disc $D:= D(z_0, \frac12 d_0)$. Let $\al_0=\al(z_0)$.
 On $D$ we have $|e^{\al_0} u-v|\lesssim d_0^{\delta+1} +d_0^\gamma \le 2 d_0^\gamma$.  
 As $e^{\al_0} u-v|$ is harmonic on $D$ we get its gradient estimated:
 $$
 |\nabla (e^{\al_0} u-v|)(z_0)|\le C d_0^{\gamma-1}.
 $$
 On the other hand $|\nabla \big[(e^\al -e^{\al_0}) v\big]| \lesssim d^{-1+\delta}\cdot d + d^\delta \le 2 d^\delta$.
 \end{proof}
 
 Put $\Phi:= e^\al v-u$, $\pd \Phi = e^\al\pd v -\pd u  + v\pd e^\al$, so, by Lemma \ref{2} with $\gamma=\delta+\eps$ we have
 \begin{equation}
 \label{pdPhi}
 |\pd\Phi | \le \frac1{d^{1-\delta-\eps}}\,.
 \end{equation}
 
 \medskip
 
 Now we write 
 \begin{eqnarray}
 \label{int1}
 & \pd \Phi =\int_{\pd U}\frac{\pd \Phi(\zeta)}{z-\zeta} d\zeta + \int\int_{U\setminus D_k} \frac{\bar\pd \pd \Phi(\zeta)}{z-\zeta} dm_2(\zeta)+\notag
 \\
 & \int_{\pd D_k}\frac{\pd \Phi(\zeta)}{z-\zeta} d\zeta =: I +II +III\,.
 \end{eqnarray}
 By \eqref{pdPhi} $III| \le Cm_k\cdot a^{-k}\cdot a^{-k(-1+delta+\eps)}= a^{-k\eps}\to 0$, $k\to \infty$.
 
 On the other hand, $|\bar\pd \pd \Phi|= |\bar\pd \pd e^{al} \cdot v  +\pd e^\al\cdot \bar\pd v +
 \bar\pd e^\al\cdot \pd v | \le \frac1{d^{1-\delta}}$. We use Lemma \ref{L} and Sobolev inequality to get 
 \begin{equation}
 \label{La}
 \pd\Phi\in \Lambda^{\frac{\delta}{2+\delta}},
 \end{equation}
 where $\Lambda^s$ stands for the H\"older class of order $s$. This is because $\frac1\zeta \,*\, L^p \subset \Lambda^{1-\frac2p}$.
 But $\Phi$ is real-valued, so \eqref{La} implies that $\nabla \Phi |(U\setminus J) \in \Lambda^{\frac{\delta}{2+\delta}}(U) |(U\setminus J)$.
Repeat the integral formula consideration \eqref{int1}, but now for $\Phi$ itself.
 As $|\Phi|\lesssim d^\delta$ we will get
 \begin{equation}
 \label{int2}
 \forall z\in U\setminus J,\,\, \,\Phi(z) = \int_{\pd U}\frac{ \Phi(\zeta)}{z-\zeta} d\zeta + \int\int_{U\setminus J} \frac{\bar\pd  \Phi(\zeta)}{z-\zeta} dm_2(\zeta)\,.
 \end{equation}
 Notice that the right hand side belongs to $C^{1+\frac{\delta}{2+\delta}} (U)$ by \eqref{La} and the fact that $\frac1{\zeta^2}\, *\, \Lambda^s \subset \Lambda^s$ if $s\in (0, 1)$.  Also $\Phi$ is continuous up to $J$.  Hence formula \eqref{int2} holds for all $z\in U$.
 
 Consequently $\Phi | U\in C^{1+\frac{\delta}{2+\delta}} (U)$. Therefore, using that $\Phi=0$ on $J$ we get
 \begin{equation}
 \label{1plus}
 |e^\al v -u|=|\Phi|\le C\,d^{1+\frac{\delta}{2+\delta}}\,.
 \end{equation}
 Notice that this inequality is an improvement upon inequality \eqref{better}. Use Lemma \ref{2} again, this time with $\gamma= 1+\frac{\delta}{2+\delta}$. The we get
\begin{equation}
 \label{nabla}
|\nabla \Phi|= |\nabla (e^\al  v -u)| \lesssim d^{\frac\delta{2+\delta}}\Rightarrow   | e^{\al(x)} \nabla  v(x) -
\nabla u(x)| \lesssim d^{\frac\delta{2+\delta}},\quad\forall x\in [-1, 1].
 \end{equation}
 By symmetry $u_y= v_y=0$ on $ [-1, 1]$. Hence
 \begin{equation}
 \label{dx}
 | e^{\al(x)}   \frac{\pd v}{\pd x} -
 \frac{\pd u}{\pd x}| \lesssim d^{\frac\delta{2+\delta}},\quad\forall x\in [-1, 1].
 \end{equation}
 
 Introduce a new function holomorphic in $U\setminus J$:
 $$
 H(z):= A(z) \frac{\pd v}{\pd z}  - \frac{\pd u}{\pd z}
 $$
 We want to prove that it is holomorphic in $U$. Its restriction to $[-1,1]$ satisfies \eqref{dx}. 
 So if we prove that it is  indeed holomorphic in $U$ we will prove that it is identically zero.
 
 We know that a priori
 \begin{equation}
 \label{growth}
 |H(x+iy)| \le \frac{C}{d(x+iy)^{1-\delta}} \le  \frac{C}{|y|^{1-\delta}}\,.
 \end{equation}
 
 \subsection{A punch line}
 \label{punch}
 Let $\bD$ denote the unit disc.
 \begin{lemma}[Hruschev, 1976]
 \label{Hruschev}
 Let $f$ be holomorphic in $\bD_+= \bD\cap \bC_+$ and let it satisfy
 \begin{enumerate}
 \item   $|f(z)| \le C_1 e^{\frac{c_2}{|\Im z|^\sigma}},\quad \sigma \in (0, 1)$;
 \item  for a $J\subset \bR$, $f$ has boundary values $f^* (x) =\lim_{z\to x, z\in \bC_+} f(z)$, $x\in \bR\setminus J$;
 \item  $|f^*(x)| \le 1$,  for all $x\in \bR\setminus J$;
 \item  $\cH^{1-\sigma}(J) =0$.
 \end{enumerate}
 Then $|f(z)|\le 1$ on $\bD_+$.
 \end{lemma}
 This Lemma of S. Hruschev \cite{SH} is a sophisticated variant of the Phragm\'en--Lindel\"of principle.
 In our case of regular Cantor set $J$ we can choose any $1-\delta<\sigma<1$ to get
 $$
 \cH^{1-\sigma} (J)=0\,.
 $$
 All other assumptions hold with a good margin. Thus, our $H$ is bounded in $U_+=U\cap\bC_+$ 
 (recall that $U$ is a neighborhood of $[-1,1]\supset J$). Similarly $H$ is bounded in $U_-=U\cap\bC_-$. 
 But the boundary values of $H$ from $U_+$ coincide with the boundary values of $H$ from 
 $U_-$ on $[-1,1]\setminus J$ because $H$ is holomorphic in $U\setminus J$. 
 So, those boundary values coincide Lebesgue almost everywhere on $[-1,1]$. 
 We also just established that $H$ is bounded in $U\setminus J\supset U\setminus [-1,1]$.
 Therefore, $H$ removes the singularity $J$, meaning that $H$ is holomorphic in the whole $U$.
 
 Now we can conclude from \eqref{dx} combined with $u_y= v_y=0$ on $ [-1, 1]$ that $H$ vanishes on $J$. Hence,
 \begin{equation}
 \label{Hv}
 H\equiv 0\Rightarrow e^{\al(x)} \frac{\pd v}{\pd x} -\frac{\pd u}{\pd x}\equiv 0,\quad x\in [-1,1]\setminus J\,.
 \end{equation}
 
 \medskip
 
 This is almost the end of the story, but not quite the end.  Consider again 
 $$
 \Phi|([-1,1]\setminus J) = e^{\al(x)} v(x)- u(x).
 $$
 By the vanishing assumption of $u,v$ on $J$ it vanishes at the end points of every complimentary interval $\ell$ of $J$ in $[-1,1]$.
So, for every complimentary interval $\ell$ there exists a point $c_\ell\in \ell$ such that
$$
\frac{d}{dx}(e^\al v-u)(c_\ell)=0\,.
$$
Combined with \eqref{Hv} this gives us
$$
\frac{d}{dx}(e^\al)(c_\ell)=0, \quad \forall \ell\,.
$$
But $\al$ is a real analytic  function and there are infinitely many $c_\ell$ ($J$ is a Cantor set). Therefore,
$$
A=e^\al\equiv \const\,.
$$
Hence,
$$
\frac{u(x)}{v(x)} = \la=\const, \,\, x\in J.
$$
From here it follows easily that $u-\la v$ is harmonic not only in $U\setminus J$ but in $U$ as a whole. So $u-\la v$ is real analytic on $[-1,1]$. Being zero on $J$ this harmonic function  must be zero on real analytic arc containing $J$, thus
$$
u-\la v\equiv 0 \,\,\text{on}\,\, [-1,1]\,.
$$
Inevitably
$$
 \frac{\pd v}{\pd x} -\la\frac{\pd u}{\pd x}\equiv 0,\quad x\in [-1,1].
 $$
 At the same time, by symmetry $ \frac{\pd v}{\pd y} -\la\frac{\pd u}{\pd y}\equiv 0,\quad x\in [-1,1]$. Hence, $ \frac{\pd v}{\pd z} -\la\frac{\pd u}{\pd z}\equiv 0$ on $[-1,1]$, and so, in $U$. Function $u-\la v$ is real valued, so the same is true for $\bar\pd$-derivative, and this means that $u-\la v\equiv\const$ in $U$. But it vanishes on $J$. Hence,
 $$
 u-\la v \equiv 0\,.
 $$
 Theorem \ref{Rth} is completely proved.
\end{proof}

\end{document}